\documentclass[a4paper]{amsart}
\usepackage{tikz-cd}
\usepackage{ifthen}
\usepackage{amssymb}
\usepackage{amsthm}
\usepackage{hyperref}
\calclayout 

\title{Highest weight categories and recollements}
\author[Henning Krause]{Henning Krause}
\address{Henning Krause\\ Fakult\"at f\"ur Mathematik\\
Universit\"at Bielefeld\\ 33501 Bielefeld\\ Germany.}
\email{hkrause@math.uni-bielefeld.de}
\date{December 15, 2015}

\theoremstyle{plain}
\newtheorem{lem}{Lemma}[section]
\newtheorem{prop}[lem]{Proposition}
\newtheorem{cor}[lem]{Corollary}
\newtheorem{thm}[lem]{Theorem}

\theoremstyle{remark}
\newtheorem{rem}[lem]{Remark}

\theoremstyle{definition}
\newtheorem{exm}[lem]{Example}
\newtheorem{defn}[lem]{Definition}

\numberwithin{equation}{section}

\newcommand{\smatrix}[1]{\left[\begin{smallmatrix}#1\end{smallmatrix}\right]}

\renewcommand{\mod}{\operatorname{mod}\nolimits}

\newcommand{\ann}{\operatorname{ann}\nolimits}

\newcommand{\proj}{\operatorname{proj}\nolimits}

\newcommand{\rad}{\operatorname{rad}\nolimits}

\newcommand{\add}{\operatorname{add}\nolimits}

\newcommand{\coh}{\operatorname{coh}\nolimits}
\newcommand{\Mod}{\operatorname{Mod}\nolimits}

\newcommand{\End}{\operatorname{End}\nolimits}
\newcommand{\Hom}{\operatorname{Hom}\nolimits}

\renewcommand{\Im}{\operatorname{Im}\nolimits}
\newcommand{\Ker}{\operatorname{Ker}\nolimits}

\renewcommand{\dim}{\operatorname{dim}\nolimits}

\newcommand{\Ext}{\operatorname{Ext}\nolimits}
\newcommand{\Tor}{\operatorname{Tor}\nolimits}
\newcommand{\Filt}{\operatorname{Filt}\nolimits}
\newcommand{\Serre}{\operatorname{Serre}\nolimits}

\newcommand{\Pic}{\operatorname{Pic}\nolimits}

\newcommand{\id}{\operatorname{id}\nolimits}

\newcommand{\op}{\mathrm{op}}

\newcommand{\lto}{\longrightarrow}
\renewcommand{\to}{\rightarrow}
\newcommand{\xto}{\xrightarrow}

\def\e{\varepsilon}

\def\p{\phi}

\def\De{\Delta}
\def\Ga{\Gamma}
\def\La{\Lambda}

\def\A{{\mathcal A}}
\def\B{{\mathcal B}}
\def\C{{\mathcal C}}

\def\O{{\mathcal O}}
\def\P{{\mathcal P}}

\def\bbQ{\mathbb Q}
\def\bbR{\mathbb R}
\def\bbX{\mathbb X}

\def\bfD{\mathbf D}

\def\fra{\mathfrak a}

\begin{document}

\begin{abstract}
  We provide several equivalent descriptions of a highest weight
  category using recollements of abelian categories. Also, we explain
  the connection between sequences of standard and exceptional
  objects.
\end{abstract}

\maketitle 

\section{Introduction}

Highest weight categories and quasi-hereditary algebras arise
naturally in representation theory and were introduced in a series of
papers by Cline, Parshall, and Scott \cite{CPS1988, PS1988,Sc1987};
see also the work of Dlab and Ringel \cite{DR1989,DR1992}. The
intimate connection between highest weight categories and recollements
of derived categories was noticed right from the beginning. In this
note we \emph{characterise} highest weight categories in terms of
recollements of abelian categories; see Theorem~\ref{th:hwt}.

A highest weight category is determined by its standard objects
(usually denoted by $\De_i$, where the index $i$ refers to the
weight).  An efficient way to formulate this for a module category is
given by the following result, which is a variation of a result of
Dlab and Ringel \cite{DR1992}. 

\begin{thm}\label{th:standard-defn}
  Let $\A$ be the category of finitely generated modules over an artin
  algebra. Then $\A$ is a highest weight category if and only if there
  are objects $\De_1,\ldots,\De_n$  having the following
  properties:
\begin{enumerate}
\item $\End_\A(\De_i)$ is a division ring for all $i$.
\item $\Hom_\A(\De_i,\De_j)=0$ for all $i> j$. 
\item $\Ext^1_\A(\De_i,\De_j)=0$ for all $i\ge  j$. 
\item A projective generator of $\Filt(\De_{1},\ldots,\De_n)$ is also
  one for $\A$.\qed
\end{enumerate}
\end{thm}

This description of a highest weight category via its sequence of
standard objects suggests a close connection with the concept of an
exceptional sequence, as introduced in the study of vector bundles
\cite{Bo1990,Go1988,GR1987,Ru1990}.  We make this connection precise
in Theorem~\ref{th:exceptional} and claim that both concepts are
basically equivalent, even though their origins are quite different. A
special instance of this theorem for vector bundles on rational
surfaces is due to Hille and Perling \cite{HP2014}. For further
examples of this connection, relating derived categories of
Grassmannians and modular representation theory, see \cite{BLV,
  Ef2014}.

The crucial issue for understanding the concept of a highest weight
category is to find out when a recollement of abelian categories
extends to a recollement of their derived categories. We address this
problem explicitly in an appendix and provide a necessary and
sufficient criterion. This is not used in the main part of the paper
but serves as an illustration for some of the key arguments and might
be of independent interest.

This paper is organised as follows. In \S\ref{se:recollements} we
recall definitions and basic facts about recollements of abelian and
triangulated categories. The characterisation of highest weight
categories via recollements is given in \S\ref{se:hwt}. Then we
explain in \S\ref{se:quasi-hereditary} the equivalent concept of a
quasi-hereditary ring and provide a method for constructing
quasi-hereditary endomorphism rings in abelian categories. The final
\S\ref{se:exceptional} is devoted to the connection between sequences
of standard and exceptional objects.

Polynomial representations of general linear groups provide interesting
examples of highest weight categories. In that case it is appropriate
to work with $k$-linear highest weight categories over an arbitrary
commutative base ring $k$, and we refer to \cite{Kr2014} for a
detailed exposition.

\section{Recollements}\label{se:recollements}

\subsection*{Recollements of abelian and triangulated categories}
We recall the definition of a recollement using the standard notation
\cite[1.4]{BBD1982}. In fact, any recollement is built from two
diagrams involving `localisation' \cite{Ga1962} and `colocalisation'
\cite{Se1960}.

\begin{defn}
  A \emph{localisation sequence} of abelian (triangulated) categories
  is a diagram of functors
\begin{equation}\label{eq:loc}
\begin{tikzcd}
\A' \arrow[tail,yshift=0.75ex]{rr}{i_!} &&\A  \arrow[twoheadrightarrow,yshift=-0.75ex]{ll}{i^!}
\arrow[twoheadrightarrow,yshift=0.75ex]{rr}{j^*} &&\A''  \arrow[tail,yshift=-0.75ex]{ll}{j_*}
\end{tikzcd}
\end{equation}
satisfying the following conditions:
\begin{enumerate}
\item  $i_!$ and $j^*$ are exact functors of  abelian (triangulated) categories.
\item $(i_!,i^!)$ and $(j^*,j_*)$ are adjoint pairs.
\item $i_!$ and $j_*$ are fully faithful functors.
\item An object in $\A$ is annihilated by $j^*$ iff it
  is in the essential image of $i_!$.
\end{enumerate}
\end{defn}
Note that condition (3) admits an equivalent formulation; see
\cite[I.1.3]{GZ}.  In the presence of (2), the functor $i_!$ is fully
faithful iff the unit $\id_{\A'}\to i^!i_!$ is an isomorphism. Also,
the functor $j_*$ is fully faithful iff the counit
$j^*j_*\to\id_{\A''}$ is an isomorphism.

\begin{defn}  
A \emph{colocalisation sequence} of abelian (triangulated) categories
  is a diagram of functors
\begin{equation}\label{eq:coloc}
\begin{tikzcd}
\A' \arrow[tail,yshift=-0.75ex]{rr}[swap]{i_*} &&\A  \arrow[twoheadrightarrow,yshift=0.75ex]{ll}[swap]{i^*}
\arrow[twoheadrightarrow,yshift=-0.75ex]{rr}[swap]{j^!} &&\A''  \arrow[tail,yshift=0.75ex]{ll}[swap]{j_!}
\end{tikzcd}
\end{equation}
satisfying the following conditions:
\begin{enumerate}
\item  $i_*$ and $j^!$ are exact functors of  abelian (triangulated) categories.
\item $(i^*,i_*)$ and $(j_!,j^!)$ are adjoint pairs.
\item $i_*$ and $j_!$ are fully faithful functors.
\item An object in $\A$ is annihilated by $j^!$ iff it
  is in the essential image of $i_*$.
\end{enumerate}
\end{defn}

\begin{defn}
  A \emph{recollement} of abelian (triangulated) categories is a
  diagram of functors
\begin{equation}\label{eq:rec}
\begin{tikzcd}
  \A' \arrow[tail]{rr}[description]{i_*=i_!} &&\A
  \arrow[twoheadrightarrow,yshift=-1.5ex]{ll}{i^!}
  \arrow[twoheadrightarrow,yshift=1.5ex]{ll}[swap]{i^*}
  \arrow[twoheadrightarrow]{rr}[description]{j^!=j^*} &&\A''
  \arrow[tail,yshift=-1.5ex]{ll}{j_*}
  \arrow[tail,yshift=1.5ex]{ll}[swap]{j_!}
\end{tikzcd}
\end{equation}
such that the subdiagram \eqref{eq:loc} is a localisation sequence
and the subdiagram \eqref{eq:coloc} is a colocalisation sequence.

The recollement 
is called \emph{homological} if the functor $i_*$ induces for all
$X,Y\in\A'$ and $p\ge 0$ isomorphisms
\[\Ext^p_{\A'}(X,Y)\xto{\sim}\Ext^p_\A(i_*(X),i_*(Y)).\]
The  terminology follows that used   in \cite{Ps2014}, where $i_*$ is
called \emph{homological embedding}.
\end{defn}


Given a colocalisation sequence \eqref{eq:coloc} and an object $X$ in
$\A$, we have the counit $j_!j^!(X)\to X$ and the unit
$X\to i_*i^*(X)$.  These fit into an exact sequence
\[j_!j^!(X)\lto X\lto i_*i^*(X)\lto 0\qquad \text{($\A$ abelian)}\]
and an exact triangle
  \[j_!j^!(X)\lto X\lto i_*i^*(X)\lto \qquad \text{($\A$ triangulated)}.\] 

  Often we consider abelian categories having \emph{enough projective
    objects}, that is, every object $X$ admits an epimorphism $P\to X$
  with $P$ projective. We use without mentioning that a left adjoint
  of an exact functor preserves projectivity.

\subsection*{Recollements of module categories}

Let $\La$ be a ring (associative with identity). We consider the
category $\Mod\La$ of right $\La$-modules.  We write $\mod\La$ for the
full subcategory of finitely presented $\La$-modules and $\proj\La$
for the full subcategory of finitely generated projective
$\La$-modules.

The following result summarises some basic facts about subcategories
of $\Mod\La$ consisting of modules that are annihilated by a fixed
ideal. Note that all ideals in this work are two-sided.

Recall that a full subcategory $\C\subseteq\A$ of an abelian category
is a \emph{Serre subcategory} if for every exact sequence
$0\to X'\to X\to X''\to 0$ in $\A$ we have $X\in\C$ iff
$X',X''\in\C$. For example, the objects that are annihilated by an
exact functor $\A\to\A'$ form a Serre subcategory.

\begin{prop}[{\cite[Proposition~7.1]{Au1974}}]
\label{pr:annihilated}
  A full subcategory $\C$ of $\Mod\La$ is of the form $\Mod\La/\fra$
  for some ideal $\fra$ of $\La$ if and only if the
  following holds:
\begin{enumerate}
\item If $X'\subseteq X$ is a submodule of $X\in\C$, then $X'$ and
  $X/X'$ are in $\C$.
\item If $(X_i)_{i\in I}$ is a family of modules in $\C$, then their product
  $\prod_{i\in I}X_i$ is in $\C$.
\end{enumerate}
In this case $\fra=\bigcap_{X\in\C}\ann X$. Moreover, $\fra^2=\fra$ if
and only if $\C$ is a Serre subcategory.\qed
\end{prop}

Given an idempotent $e\in\La$, the inclusion
$i_*\colon \Mod \La/\La e\La\to\Mod\La$ and 
\[j^*:=\Hom_\La(e\La,-)\cong -\otimes_\La \La e\]
induce a recollement
\begin{equation}\label{eq:idempotent}
\begin{tikzcd}
  \Mod \La/\La e\La\arrow[tail]{rr}[description]{i_*} &&\Mod\La
  \arrow[twoheadrightarrow,yshift=-1.5ex]{ll}
  \arrow[twoheadrightarrow,yshift=1.5ex]{ll}
  \arrow[twoheadrightarrow]{rr}[description]{j^*} &&\Mod e\La e\, .
  \arrow[tail,yshift=-1.5ex]{ll}
  \arrow[tail,yshift=1.5ex]{ll}
\end{tikzcd}
\end{equation}
In fact, any recollement of module categories 
\[\begin{tikzcd}
  \Mod \La'\arrow[tail]{rr} &&\Mod\La
  \arrow[twoheadrightarrow,yshift=-1.5ex]{ll}
  \arrow[twoheadrightarrow,yshift=1.5ex]{ll}
  \arrow[twoheadrightarrow]{rr} &&\Mod \La''
  \arrow[tail,yshift=-1.5ex]{ll}
  \arrow[tail,yshift=1.5ex]{ll}
\end{tikzcd}
\]
is up to Morita equivalence of this form.  For each $\La$-module $X$
there is a natural exact sequence
\begin{equation}\label{eq:counit-idempotent}
  \Hom_\La(e\La,X)\otimes_{e\La e} e\La\xto{\ \e_X\ }
  X\lto  X\otimes_\La \La/\La e\La \lto  0.
\end{equation}
If $\La$ is right artinian, then \eqref{eq:idempotent} restricts to a
colocalisation sequence
\[
\begin{tikzcd}
  \mod \La/\La e\La\arrow[tail,yshift=-.75ex]{rr}&&\mod\La
  \arrow[twoheadrightarrow,yshift=.75ex]{ll}
   \arrow[twoheadrightarrow,yshift=-.75ex]{rr} &&\mod e\La e\, .
  \arrow[tail,yshift=.75ex]{ll}
 \end{tikzcd}
\]

\begin{lem}\label{le:Serre}
  Let $\La$ be a right artinian ring and $\C\subseteq\mod\La$ a Serre
  subcategory. Then there is an idempotent $e\in\La$ such that
  $\C=\mod \La/\La e\La$.  Moreover, the following holds:
\begin{enumerate}
\item The inclusion $\mod \La/\La e\La\to\mod\La$ admits a left and a right adjoint.
\item The functor $\Hom_\La(e\La,-)\colon\mod\La\to\mod e\La e$  admits a left adjoint. 
\item The functor $\Hom_\La(e\La,-)\colon\mod\La\to\mod e\La e$
  admits a right adjoint provided that $\mod\La$ has enough injective objects.
\end{enumerate}
\end{lem}
\begin{proof}
  The annihilator $\fra\subseteq\La$ of the modules in $\C$ is
  idempotent since $\C$ is closed under forming extension. Thus
  $\fra=\La e\La$ for some idempotent $e\in\La$.

(1) The right adjoint sends a $\La$-module $X$ to the maximal
submodule belonging to $\C$. The left adjoint sends  $X$ to the maximal
factor module belonging to $\C$.

(2) Take $-\otimes_{e\La e} e\La$.

(3) Let $E$ be an injective cogenerator and set
$\Ga=\End_\La(E)^\op$. Then we have $(\mod\La)^\op\xto{\sim}\mod\Ga$
via $\Hom_\La(-,E)$ and can apply (2).
\end{proof}

\begin{exm}
  Consider the right artinian ring $\La=\smatrix{\bbR&\bbR\\ 0&\bbQ}$
  and $e=\smatrix{0&0\\ 0&1}$. Then
  $\Hom_\La(e\La,-)\colon\mod\La\to\mod e\La e$ admits no right
  adjoint, because it would send $e\La e$ to $\Hom_{e\La e}(\La e,e\La e)$ which is not finitely generated
  over $\La$.
\end{exm}

We recall a well known criterion for a recollement of module
categories to be homological.

\begin{lem}\label{le:idempotent}
  Let $\La$ be a ring and $\fra\subseteq\La$ an ideal. Then the
  following are equivalent:
\begin{enumerate}
\item  $\La/\fra\otimes_\La\La/\fra\cong\La/\fra$ and
  $\Tor^\La_p(\La/\fra,\La/\fra)=0$ for all $p>0$.
\item $\Ext^p_{\La/\fra}(X,Y)\xto{\sim}\Ext^p_{\La}(X,Y)$ for all
  $\La/\fra$-modules $X,Y$ and $p\ge 0$.
\end{enumerate}
These conditions are satisfied when $\fra$ is a projective $\La$-module.
\end{lem}
\begin{proof}
  For the first part, see \cite[Theorem~4.4]{GL1991}. Now suppose that
  $\fra$ is projective. This implies $\Tor^\La_*(\fra,\La/\fra)=0$, and the
  exact sequence $0\to\fra\to\La\to\La/\fra\to 0$ induces an
  isomorphism $\Tor^\La_*(\La/\fra,\La/\fra)\cong\La/\fra$. Thus (1) holds.
\end{proof}

\subsection*{Abelian length categories}

Let $\A$ be an abelian \emph{length category}. Thus $\A$ is an abelian
category and every object in $\A$ has a finite composition series.

Recall that $\A$ is
\emph{Ext-finite} if for every pair of simple objects $S$ and $T$
\[\dim_{\End_\A(T)^\op}\Ext_\A^1(S,T)<\infty.\] 
This property is useful for constructing projective generators.

\begin{prop}[{\cite[8.2]{Ga1973}}]
\label{pr:Gabriel}
An abelian length category $\A$ is equivalent to the category $\mod\La$ of finitely generated
$\La$-modules for some right artinian ring $\La$ if and only if the
following holds:
\begin{enumerate}
\item $\A$ has only finitely many simple objects.
\item $\A$ is Ext-finite.
\item The supremum of the Loewy lengths of the objects in $\A$ is
  finite.\qed
\end{enumerate}
\end{prop}

\section{Highest weight categories}\label{se:hwt}

Highest weight categories were introduced by Cline, Parshall, and
Scott \cite{CPS1988} in the context of $k$-linear categories over a
field $k$. The definition given here uses a slightly different
formulation which follows Rouquier \cite{Ro2008}.  Also, our definition is more
general since the endomorphism ring of a standard object can be any
division ring. For simplicity, we restrict ourselves to the case that
the set of weights is finite and totally ordered.

Let $\De_1,\ldots,\De_n$ be objects in an abelian category $\A$.  We
write $\Filt(\De_1,\ldots,\De_n)$ for the full subcategory of objects
$X$ in $\A$ that admit a finite filtration
$0=X_0\subseteq X_1\subseteq\ldots\subseteq X_t=X$ such that each
factor $X_i/X_{i-1}$ is isomorphic to an object of the form
$\De_j$. Also, let $\Serre(\De_1,\ldots,\De_n)$ denote the smallest
Serre subcategory of $\A$ containing $\De_1,\ldots,\De_n$.

Recall that a projective object $P$ of an abelian (or exact) category is a
\emph{projective generator} if for every object $X$ there is an
exact sequence $0\to N\to P^r\to X\to 0$ for some positive integer
$r$.

\begin{defn}\label{de:hwt}
  Let $\A$ be an abelian length category having only finitely many
  isoclasses of simple objects. Then $\A$ is called \emph{highest
    weight category} if there are finitely many exact sequences
\begin{equation}\label{eq:hwt}
0\lto U_i\lto P_i\lto \De_i \lto 0\qquad (1\le i\le n)
\end{equation}
  in $\A$ satisfying the following:
\begin{enumerate}
\item $\End_\A(\De_i)$ is a division ring for all $i$.
\item $\Hom_\A(\De_i,\De_j)=0$ for all $i> j$. 
\item $U_i$ belongs to $\Filt(\De_{i+1},\ldots,\De_n)$ for all $i$.
\item $\bigoplus_{i=1}^n P_i$ is a projective generator of $\A$.
\end{enumerate}
The objects $\De_1,\ldots,\De_n$ are called \emph{standard objects}. 
\end{defn}

Now fix a highest weight category $\A$ with standard objects
$\De_1,\ldots,\De_n$. Let $P$ denote a projective generator and set
$\La=\End_\A(P)$. We identify $\A=\mod\La$ via $\Hom_\A(P,-)$.  Set
$\Ga=\End_\La(\De_n)$ and note that $\De_n$ is projective. For each
$\La$-module $X$ there is a natural exact sequence
\begin{equation}\label{eq:counit}
  \Hom_\La(\De_n,X)\otimes_\Ga \De_n\xto{\ \e_X\ } X\lto \bar X\lto 0.
\end{equation}

Note that $\Ker\e_X$ and $\bar X$ are annihilated by
$\Hom_\La(\De_n,-)$ since $\Hom_\La(\De_n,\e_X)$ is invertible. The
homomorphism $\La\to \bar\La$ identifies via restriction of scalars
\[\mod\bar\La=\{X\in \mod\La\mid\Hom_\La(\De_n,X)=0\}\]
and $\bar X\cong X\otimes_\La\bar\La$ for all $X\in\mod\La$.

\begin{lem}\label{le:filtered}
\begin{enumerate}
\item The counit $\e_X$ is a monomorphism for $X$ in $\Filt(\De_{1},\ldots,\De_n)$. 
\item The assignment $X\mapsto \bar X$ provides an exact left adjoint of
the inclusion $\Filt(\De_{1},\ldots,\De_{n-1})\to \Filt(\De_{1},\ldots,\De_n)$.
\end{enumerate}
\end{lem}
\begin{proof}
  An induction on the length of a filtration of an object $X$ in
  $\Filt(\De_{1},\ldots,\De_n)$ yields some $r\ge 0$ and an exact
  sequence $0\to\De_n^r\to X\to X'\to 0$ with $X'$ in
  $\Filt(\De_{1},\ldots,\De_{n-1})$. Then we have
  $\Hom_\La(\De_n,X)\otimes_\Ga \De_n\cong\De_n^r$ and
  $\bar X\cong X'$. The exactness follows from the snake lemma since
  $\Hom_\La (\De_n ,-) \otimes_\Ga\De_n$ is exact.
\end{proof}

\begin{lem}\label{le:reduction}
  Let $\A$ be a highest weight category with standard objects
  $\De_1,\ldots,\De_n$. For the full subcategory
  $\bar\A=\{X\in \A\mid\Hom_\A(\De_n,X)=0\}$ the following holds:
\begin{enumerate}
\item $\bar\A$ is a highest weight category with standard objects
  $\De_1,\ldots,\De_{n-1}$.
\item The inclusion $\bar\A\to\A$ induces isomorphisms
$\Ext^p_{\bar\A}(X,Y)\xto{\sim}\Ext^p_{\A}(X,Y)$ for all $X,Y$ in
$\bar\A$ and $p\ge 0$.
\end{enumerate}
\end{lem}
\begin{proof} (1)
Applying the assignment $X\mapsto \bar X$ to \eqref{eq:hwt}
  yields exact sequences
\begin{equation}\label{eq:hwt2} 
0\lto \bar U_i\lto \bar P_i\lto\De_i \lto 0\qquad (1\le i\le n-1)
\end{equation}
with $\bar U_i$ in $\Filt(\De_{i+1},\ldots,\De_{n-1})$ by Lemma~\ref{le:filtered}.  It
remains to observe that $\bigoplus_{i=1}^{n-1} \bar P_i$ is a
projective generator of $\bar\A$.

(2) Let $\fra$ denote the kernel of $\La\to\bar\La$. This is a
projective $\La$-module because it is a direct sum of copies of
$\De_n$ by Lemma~\ref{le:filtered}. Thus the assertion follows from
Lemma~\ref{le:idempotent}. Alternatively, use
Proposition~\ref{pr:derived}.
\end{proof}

The following result establishes the precise connection between
highest weight categories and recollements of abelian categories with
semisimple factors.  For a similar result involving recollements of
derived categories, see \cite[Theorem~5.13]{PS1988}.

\begin{thm}\label{th:hwt}
  Let $\A$ be an abelian length category with finitely many simple
  objects. Suppose that $\A$ and $\A^\op$ are Ext-finite. Then the
  following are equivalent:
\begin{enumerate}
\item The category $\A$ is a highest weight category.
\item There is a finite chain of full subcategories\[0=\A_0\subseteq \A_1\subseteq
  \ldots\subseteq\A_n=\A\]
and a sequence of division rings $\Ga_1,\ldots,\Ga_n$ such that each
inclusion $\A_{i-1}\to\A_i$ induces a homological recollement of abelian categories
\begin{equation}\label{eq:rec_hwt}
\begin{tikzcd}
  \A_{i-1} \arrow[tail]{rr}&&\A_i
  \arrow[twoheadrightarrow,yshift=-1.5ex]{ll}
  \arrow[twoheadrightarrow,yshift=1.5ex]{ll}
  \arrow[twoheadrightarrow]{rr} &&\mod \Ga_i\, .
  \arrow[tail,yshift=-1.5ex]{ll}
  \arrow[tail,yshift=1.5ex]{ll}
\end{tikzcd}
\end{equation}
\end{enumerate}
In that case the standard objects $\De_1,\ldots,\De_n$ in $\A$ are
obtained by applying the left adjoint $\mod\Ga_i\to\A_i$ to $\Ga_i$
for $1\le i\le n$. Conversely, the subcategories $\A_i\subseteq\A$ are
obtained by setting $\A_i=\Serre(\De_1,\ldots,\De_i)$ or recursively
$\A_{i-1}=\{X\in\A_i\mid \Hom_\A(\De_i,X)=0\}$.
\end{thm}

\begin{rem} 
(1) The assertion of Theorem~\ref{th:hwt} remains true if one requires
each $\Ga_i$ to be a semisimple ring.

(2) The number $n$ in Theorem~\ref{th:hwt} equals the number of
pairwise non-isomorphic simple objects in $\A$ and the $\Ga_i$ are
their endomorphism rings.

(3) Each recollement \eqref{eq:rec_hwt} restricts to a diagram of
exact functors
\[
\begin{tikzcd}
  \Filt(\De_1,\ldots,\De_{i-1}) \arrow[tail,
  yshift=-0.75ex]{rr}&&\Filt(\De_1,\ldots,\De_{i})
  \arrow[twoheadrightarrow,yshift=0.75ex]{ll}
  \arrow[twoheadrightarrow, yshift=-0.75ex]{rr} &&\Filt(\De_i)\, .
  \arrow[tail,yshift=0.75ex]{ll}
\end{tikzcd}
\]

(4) Each recollement \eqref{eq:rec_hwt} induces for the corresponding
bounded derived categories a recollement of triangulated categories
\[
\begin{tikzcd}
\bfD^b(\A_{i-1}) \arrow[tail]{rr}&&\bfD^b(\A_i)
  \arrow[twoheadrightarrow,yshift=-1.5ex]{ll}
  \arrow[twoheadrightarrow,yshift=1.5ex]{ll}
  \arrow[twoheadrightarrow]{rr} &&\bfD^b(\mod \Ga_i)\, .
  \arrow[tail,yshift=-1.5ex]{ll}
  \arrow[tail,yshift=1.5ex]{ll}
\end{tikzcd}
\]
This follows, for example, from \cite[Lemme~2.1.3]{Il1971}.
\end{rem}

\begin{proof}
  (1) $\Rightarrow$ (2): Suppose $\A$ is a highest weight category
  with standard objects $\De_1,\ldots,\De_n$. Observe that $\A$ has
  enough injective objects by Proposition~\ref{pr:Gabriel}, since
  $\A^\op$ is Ext-finite. We give a recursive construction of a chain
  \[0=\A_0\subseteq\A_1\subseteq\ldots\subseteq\A_n=\A\] of full
  subcategories satisfying the conditions in (2). Let $\A_{n-1}$
  denote the full subcategory of objects $X$ in $\A$ such that
  $\Hom_\A(\De_n,X)=0$ and set $\Ga_n=\End_\A(\De_n)$. The object
  $\De_n$ is projective and $\Hom_\A(\De_n,-)$ induces a recollement
\[
\begin{tikzcd}
  \A_{n-1} \arrow[tail]{rr}&&\A
  \arrow[twoheadrightarrow,yshift=-1.5ex]{ll}
  \arrow[twoheadrightarrow,yshift=1.5ex]{ll}
  \arrow[twoheadrightarrow]{rr} &&\mod \Ga_n
  \arrow[tail,yshift=-1.5ex]{ll}
  \arrow[tail,yshift=1.5ex]{ll}
\end{tikzcd}
\]
by Lemma~\ref{le:Serre}.  In Lemma~\ref{le:reduction} it is shown that
the recollement is homological and that $\A_{n-1}$ is a highest weight
category.

(2) $\Rightarrow$ (1): Fix a chain of full subcategories
$\A_i\subseteq\A$ satisfying the conditions in (2). We show by
induction on $n$ that $\A$ is a highest weight category. Let $\De_n$
denote the image of $\Ga_n$ under the left adjoint
$\mod\Ga_n\to\A$. Clearly, $\End_\A(\De_n)\cong\Ga_n$ and $\De_n$ is a
projective object. The induction hypothesis for $\A_{n-1}$ yields a
collection of exact sequences \eqref{eq:hwt2}. We modify them as
follows to obtain exact sequences \eqref{eq:hwt}.  

Fix $1\le t <n$. Observe that $\De_n/\rad\De_n$ is a simple object and
that
\[\Ext^1_\A(\bar P_t,\De_n)\xto{\sim} \Ext^1_\A(\bar P_t,\De_n/\rad\De_n)\] since
$\rad\De_n$ belongs to $\A_{n-1}$.  Using the Ext-finiteness of $\A$,
we can form the universal extension
\begin{equation}\label{eq:universal}
0\lto \De_n^r\lto P_t\lto\bar P_t\lto 0
\end{equation}
in $\A$, that is, the induced map
$\Hom_\A(\De_n^r,\De_n)\to\Ext_\A^1(\bar P_t,\De_n)$ is
surjective. This implies $\Ext^1_\A(P_t,\De_n)=0$.

We claim that $P_t$ is a projective object. First observe that for any
object $X$ in $\A$, the recollement \eqref{eq:rec_hwt} yields
an exact sequence
\begin{equation}\label{eq:co-loc}
0\lto \Ker\e_X\lto j_!j^!(X)\xto{\ \e_X\ } X\lto i_*i^*(X)\lto 0
\end{equation} 
with $\Ker\e_X$ in the image of $i_*$, since $j^!(\e_X)$ is invertible
(using the notation of \eqref{eq:rec}). The functor $\Ext^p_\A(P_t,-)$
vanishes for all $p > 0$ on the image of $i_*$ because the recollement
is homological, and $\Ext^1_\A(P_t,-)$ vanishes on the image of $j_!$
since $\Ext^1_\A(P_t,\De_n)=0$.  Now one applies the sequence
\eqref{eq:co-loc} by writing it as composite of two exact sequences
\[0\to \Ker\e_X\to j_!j^!(X)\to X'\to 0\quad\text{and}\quad 0\to X'\to X\to i_*i^*(X)\lto 0.\]
From the first sequence one gets $\Ext^1_\A(P_t,X')=0$, and then the second sequence gives $\Ext^1_\A(P_t,X)=0$.

Combining the  exact sequences \eqref{eq:hwt2} and
\eqref{eq:universal} gives for each $t$ the following commutative diagram
  with exact rows and columns.
\begin{equation*}
\begin{tikzcd}
{}&0\arrow{d}&0\arrow{d}\\
{}&\De_n^r\arrow{d}\arrow[equal]{r}&\De_n^r\arrow{d}\\
0\arrow{r}&U_t\arrow{d}\arrow{r}&P_t\arrow{d}\arrow{r}&\De_t\arrow[equal]{d}\arrow{r}&0\\
0\arrow{r}&\bar U_t\arrow{d}\arrow{r}&\bar P_t\arrow{d}\arrow{r}&\De_t\arrow{r}&0\\
&0&0
\end{tikzcd}
\end{equation*}
This yields exact sequences \eqref{eq:hwt} with $U_t$ in
$\Filt(\De_{t+1},\ldots,\De_n)$, where $P_n:=\De_n$ and $U_n:=0$. We
observe that $\bigoplus_t P_t$ is a projective generator of $\A$.

It remains to show that $\A_t=\Serre(\De_1,\ldots,\De_t)$ for
$1\le t\le n$. We prove this by induction on $t$ and using the 
recollement \eqref{eq:rec_hwt}.  For $X\in\A_t$ we have
$j_!j^!(X)=\De_t^r$ for some $r\ge 0$ and
$i_*i^*(X)\in\A_{t-1}=\Serre(\De_1,\ldots,\De_{t-1})$. Thus
$X\in\Serre(\De_1,\ldots,\De_{t})$. The other inclusion is clear.
\end{proof}

\section{Quasi-hereditary rings}\label{se:quasi-hereditary}

Quasi-hereditary rings provide an alternative concept for describing a
highest weight category.  Quasi-hereditary algebras over a field were
introduced by Scott \cite{Sc1987}; the definition given here for
semiprimary rings is due to Dlab and Ringel \cite{DR1989}.

Recall that a ring $\La$ is \emph{semiprimary} if its Jacobson radical
$J(\La)$ is nilpotent and $\La/J(\La)$ is semisimple. For example, the
endomorphism ring of an object having finite composition length is
semiprimary.

\begin{defn}
  An ideal $\fra\subseteq \La$ of a semiprimary ring $\La$ is an
  \emph{heredity ideal} if $\fra$ is idempotent, $\fra$ is a
  projective $\La$-module, and $\fra J(\La) \fra=0$.
\end{defn}

Note that an ideal $\fra$ of a semiprimary ring $\La$ is idempotent
iff there exists an idempotent $e\in\La$ such that $\fra=\La e\La$;
see \cite[Statement~6]{DR1989}. In that case $\fra J(\La) \fra=0$ iff
the ring $e\La e$ is semisimple.

\begin{defn}
  A semiprimary ring $\La$ is \emph{quasi-hereditary} if there is a
  finite sequence of surjective ring homomorphisms
\begin{equation}\label{eq:ringhom}
\La=\La_n\to\La_{n-1}\to\cdots\to\La_1\to\La_0=0
\end{equation}
such that for each $1\le i\le n$ the kernel of $\La_i\to\La_{i-1}$ is
an heredity ideal. Clearly, such a sequence is equivalent
to a finite chain of ideals
\[0=\fra_n\subseteq \ldots\subseteq  \fra_{1}\subseteq\fra_0=\La\] such that
$\fra_{i-1}/\fra_i$ is an heredity ideal in $\La/\fra_i$ for all $i$. 
\end{defn}

For $k$-linear highest weight categories over a field $k$, the
following result is due to Cline, Parshall, and Scott \cite{CPS1988}.

\begin{thm}\label{th:qhered}
  Let $\A$ be an abelian length category $\A$ having only finitely
  many isoclasses of simple objects. Then $\A$ is a highest weight
  category if and only there is a quasi-hereditary ring $\La$ such
  that $\A\xto{\sim}\mod\La$.
\end{thm}
\begin{proof}
  The proof is by induction on the number of simple objects in $\A$
  and yields an explicit correspondence between the standard objects
  in $\A$ and the chain of ideals in $\La$.

  Suppose that $\A$ is a highest weight category with standard objects
  $\De_1,\ldots,\De_n$. Then we have $\A=\mod\La$ for a ring $\La$ and
  there is a surjective homomorphism $f\colon\La\to\bar\La$ such that
  $\mod\bar\La=\{X\in \A\mid\Hom_\A(\De_n,X)=0\}$ is a highest weight
  category, by Lemma~\ref{le:reduction}. The induction hypothesis
  implies that $\bar\La$ is quasi-hereditary, and we need to show that
  $\fra:=\Ker f$ is an heredity ideal. Observe first that
  $\De_n\cong e\La$ for some idempotent $e\in\La$, and therefore
  $\fra=\La e\La$.  We have $e J(\La) e=0$ since
  $e\La e\cong\End_\La(\De_n)$ is a division ring. Moreover, $\fra$ is
  a direct sum of copies of $\De_n$, since the counit $\e_\La$ in
  \eqref{eq:counit} is a monomorphism by Lemma~\ref{le:filtered}. Thus
  $\fra$ is a projective $\La$-module.

  Now suppose that $\La$ is a quasi-hereditary ring with $\A=\mod\La$.
  Thus there is a sequence of surjective ring homomorphisms
  \eqref{eq:ringhom} such that the kernel of each $\La_i\to\La_{i-1}$
  is an heredity ideal. We may assume that $n$ is maximal. Set
  $\bar\La=\La_{n-1}$. Then the induction hypothesis implies that
  $\mod\bar\La$ is a highest weight category, say with standard
  objects $\De_1,\ldots,\De_{n-1}$, and we view this as full
  subcategory of $\mod\La$ via restriction along
  $f\colon\La\to\bar\La$. There is an idempotent $e\in\La$ such that
  $\Ker f=\La e\La$ and we set $\De_n=e\La$. Then
  $\End_\La(\De_n)\cong e\La e$ is semisimple since $e J(\La)e=0$. In
  fact, it is a divison ring because of the maximality of $n$.  The
  induction hypothesis yields a collection of exact sequences
  \eqref{eq:hwt2} in $\mod\bar\La$. We modify them exactly as in the
  second part of the proof of Theorem~\ref{th:hwt} to obtain exact
  sequences \eqref{eq:hwt}. For this construction one uses that $\A$
  is Ext-finite (holds by Proposition~\ref{pr:Gabriel}) and that
  $\Ext_{\bar\La}^p(-,-)\xto{\sim}\Ext_{\La}^p(-,-)$ for all $p\ge 0$
  (holds by Lemma~\ref{le:idempotent}). Thus $\mod\La$ is a highest
  weight category.
\end{proof}

We continue with a reformulation of the definition of a
quasi-hereditary ring which makes the concept accessible for
interesting constructions.  The basic idea is to extend the definition
of an heredity ideal to the context of additive categories.

Let $\C$ be an additive category and $\B\subseteq\C$ a full additive
subcategory. We denote by $\C/\B$ the additive category having the
same objects as $\C$ while the morphisms for objects $X,Y\in\C$ are
defined by the quotient \[\Hom_{\C/\B}(X,Y)=\Hom_\C(X,Y)/\B(X,Y)\]
modulo the subgroup $\B(X,Y)$ of morphisms that factor through an
object in $\B$.

The \emph{Jacobson radical} $J(\C)$ of an additive category $\C$ is by
definition the unique two-sided ideal of morphisms in $\C$ such that
$J(\C)(X,X)$ equals the Jacobson radical of the endomorphism ring
$\End_\C(X)$ for every object $X$ in $\C$.

\begin{defn}
  A full additive subcategory $\B\subseteq\C$ of an additive category
  $\C$ is called \emph{heredity subcategory} if $J(\B)=0$ and the
  inclusion admits a right adjoint $p\colon \C\to\B$ such that for
  each $X$ in $\C$ the counit $p(X)\to X$ is a monomorphism.
\end{defn}

For a semiprimary ring $\La$ there is a bijective correspondence
between idempotent ideals of $\La$ and certain additive subcategories
of $\proj \La$.  Next we show that this restricts to a correspondence
between heredity ideals and heredity subcategories.

For an object $X$ of an additive category let $\add X$ denote the full
subcategory of direct summands of finite direct sums of copies of $X$.

\begin{lem}\label{le:heredity}
  Let $\La$ be a semiprimary ring and set $\C=\proj\La$. The
  assignments
\[\La\supseteq\fra \longmapsto \{X\in\C\mid \Hom_\La(X,\La/\fra)=0\}\subseteq\C
\quad\text{and}\quad \C\supseteq\B\longmapsto
\B(\La,\La)\subseteq\La\]
give mutually inverse and incluson preserving bijections between the
sets of
\begin{enumerate}
\item idempotent ideals of $\La$, and 
\item strictly full and idempotent complete additive subcategories of
  $\proj\La$.
\end{enumerate}
These bijections restrict to a correspondence between heredity ideals
and heredity subcategories.
\end{lem}
\begin{proof}
  For an idempotent ideal $\fra=\La e\La$, an analysis of the
  recollement \eqref{eq:idempotent} shows that
  $\add e\La=\{X\in\C\mid \Hom_\La(X,\La/\fra)=0\}$. Conversely, any
  strictly full and idempotent complete additive subcategory  $\B\subseteq\C$
  is of the form $\B=\add e\La$ for some idempotent $e\in\La$, because
  the ring $\La$ is semiperfect. Then $\B(\La,\La)=\La e\La$.

  Now fix an ideal $\fra=\La e\La$ and a subcategory $\B=\add e\La$
  that correspond to each other.  Then $\fra J(\La)\fra=0$ if and only
  if $J(\B)=0$. Assume this property, which means that $e\La e$ is
  semisimple. The assignment
  $X\mapsto \Hom_\La(e\La,X)\otimes_{e\La e}e\La$ provides a right
  adjoint for the inclusion $\B\to\C$.  We claim that $\fra$ is a
  projective $\La$-module if and only if the counit $\e_X$ in
  \eqref{eq:counit-idempotent} is a monomorphism for all $X$ in
  $\C$. For this it suffices to consider $\e_\La$, using that its
  image equals $\fra$. If $\fra$ is projective, then $\e_\La$ is a
  monomorphism since $\Hom_\La(e\La,\Ker \e_\La)=0$. Conversely, if
  $\e_\La$ is a monomorphism, then $\fra$ belongs $\B$ and is
  therefore projective.  We conclude that $\fra$ is an heredity ideal
  if and only if $\B$ is an heredity subcategory.
\end{proof}

\begin{prop}\label{pr:qhered}
A semiprimary ring is quasi-hereditary if and only if there is a  finite chain of full additive 
subcategories
\begin{equation}\label{eq:her-chain}
0=\C_n\subseteq\ldots\subseteq \C_{1}\subseteq\C_0=\proj\La
\end{equation}
such that $\C_{i-1}/\C_i$ is an heredity subcategory of $\C/\C_i$ for
$1\le i\le n$.
\end{prop}
\begin{proof}
Apply  Lemma~\ref{le:heredity}.
\end{proof}

\begin{rem}\label{re:qhered}
  For $\La$ to be quasi-hereditary it suffices to have a chain of full
  additive subcategories \eqref{eq:her-chain} satisfying for all $i$
  the following:
\begin{enumerate}
\item $J(\C_i)(X,Y)\subseteq\C_{i+1}(X,Y)$ for all $X,Y\in\C_i$.
\item The inclusion $\C_{i+1}\to\C_i$ admits a right adjoint $p_i$
  such that the counit $p_i(X)\to X$ is a monomorphism for all $X\in\C_i$.
\end{enumerate}
\end{rem}

The following result provides a natural construction of
quasi-hereditary rings which is due to Iyama \cite{Iy2003}.

\begin{cor}
  Le $\A$ be an abelian category and suppose that every object in $\A$
  has a semiprimary endomorphism ring. Fix an object $X=X_0$ and set
  $X_{t+1}=\mathfrak r X_t$ for $t\ge 0$, where
  $\mathfrak r Y=\sum_{\p\in J(\End_\A(Y))}\Im\p$ for any object $Y$ in
  $\A$. Then the endomorphism ring of $\bigoplus_{t\ge 0}X_t$ is
  quasi-hereditary.
\end{cor}
\begin{proof}
  In \cite{Iy2003} the result is stated for modules over artin
  algebras. The same proof works in our more general setting.
  
  We apply Proposition~\ref{pr:qhered} and check the conditions of the
  subsequent remark.  Set $\C_i=\add(\bigoplus_{t\ge i}X_t)$ for $i\ge 0$
  and $\La=\End_\A(\bigoplus_{t\ge 0}X_t)$. Thus we can identify
  $\proj\La=\C_0$.  Note that $\C_i=0$ for $i\gg 0$ since
  $J(\End_\A(X))$ is nilpotent. The inclusion $\C_{i+1}\to\C_i$ admits
  a right adjoint $p_i$ given by $p_i(X_t)=X_t$ for $t>i$ and
  $p_i(X_i)=X_{i+1}$.  The counit $p_i(Y)\to Y$ is a monomorphism for
  all $Y$; it is for $Y=X_t$ the identity when $t>i$ and the inclusion
  $\mathfrak r X_i\to X_i$ when $t=i$. This follows from the fact that
  $\mathfrak r X_i\to X_i$ induces a bijection
\[\Hom_\C(X_t,\mathfrak r X_i)\xto{\sim}\Hom_\C(X_t,X_i)\]
for all $t>i$. It remains to observe that
  $J(\C_i)(X,Y)\subseteq\C_{i+1}(X,Y)$ for all $X,Y\in\C_i$ by
  construction.  
\end{proof}

\section{Exceptional sequences}\label{se:exceptional}

In Theorem~\ref{th:standard-defn} we have seen a description of
highest weight categories via standard objects that suggests a close
connection with the concept of an exceptional sequence, as introduced
in the study of vector bundles \cite{Bo1990,Go1988,GR1987,Ru1990}. We
make this connection precise, and this involves the use of derived
categories.

For an exact  category $\A$ let $\bfD^b(\A)$ denote its bounded
derived category \cite{Ke1996}.

\begin{defn}
  Let $\A$ be an abelian category. An object $E$ in $\A$ is
  \emph{exceptional} if $\End_\A(E)$ is a division ring and
  $\Ext^p_\A(E,E)=0$ for all $p>0$. A sequence of objects
  $(E_1,\ldots,E_n)$ in $\A$ is called \emph{exceptional} if each
  $E_i$ is exceptional and $\Ext^p_\A(E_i,E_j)=0$ for all $i>j$ and
  $p\ge 0$.  The sequence is \emph{full} if the objects
  $E_1,\ldots,E_n$ generate $\bfD^b(\A)$ as a triangulated category,
  and we say that the sequence is \emph{strictly full} if
  the inclusion $\Filt(E_1,\ldots,E_n)\to\A$ induces a triangle
  equivalence
\[\bfD^b(\Filt(E_1,\ldots,E_n))\xto{\ \sim\ }\bfD^b(\A).\]
\end{defn}

Note that a full exceptional sequence need not be strictly full; see
Examples~\ref{ex:kalck} and \ref{ex:broomhead}.

\begin{thm}\label{th:exceptional}
  Let $k$ be a commutative artinian ring and $\A$ a $k$-linear abelian
  category such that $\Hom_\A(X,Y)$ and $\Ext^1_\A(X,Y)$ are finitely
  generated over $k$ for all $X,Y$ in $\A$. For a sequence
  $(E_1,\ldots,E_n)$ of objects in $\A$ the following are equivalent:
\begin{enumerate}
\item The sequence $(E_1,\ldots,E_n)$ is a strictly full exceptional
  sequence.
\item There is a highest weight category $\A'$ and a triangle
  equivalence $\bfD^b(\A)\xto{\sim} \bfD^b(\A')$ that maps
  $(E_1,\ldots,E_n)$ to the sequence of standard objects in $\A'$.
\end{enumerate}
\end{thm}

A special instance of this theorem for vecor bundles on rational
surfaces is due to Hille and Perling \cite{HP2014}. For further
examples of this connection, relating derived categories of
Grassmannians and modular representation theory, see \cite{BLV,
  Ef2014}.

I am grateful to Lutz Hille for pointing out the following criterion
for an exceptional sequence to be strictly full; it is an immediate
consequence of the proof of Theorem~\ref{th:exceptional}.

\begin{rem}[Hille] 
  An exceptional sequence $(E_1,\ldots,E_n)$ in $\A$ is strictly full
  if any tilting object in $\Filt(E_1,\ldots,E_n)$ is also
  a tilting object in $\A$.
\end{rem}

Recall that an object $T$ of an exact category $\A$ is a \emph{tilting
  object} if $\Ext^p_\A(T,T)=0$ for all $p>0$ and $T$ generates
$\bfD^b(\A)$ as a triangulated category.

We need some preparations for the proof Theorem~\ref{th:exceptional}
and we begin with the following well known fact
\cite[III.2.4]{Ve1997}.

\begin{lem}\label{le:perf}
  Let $\A$ be an abelian (or exact) category with projective generator
  $P$ and set $\La=\End_\A(P)$.  The inclusion $\proj\La\to\A$
  induces a triangle equivalence $\bfD^b(\proj\La)\xto{\sim}\bfD^b(\A)$
  if all objects of $\A$ have finite projective dimension.\qed
\end{lem}

\begin{lem}\label{le:filt}
Let $\A$ be a highest weight category with standard objects
$\De_1,\ldots,\De_n$. Then the inclusion
$\Filt(\De_1,\ldots,\De_n)\to\A$ induces a triangle equivalence
\[\bfD^b(\Filt(\De_1,\ldots,\De_n))\xto{\ \sim\ }\bfD^b(\A).\]
\end{lem}
\begin{proof}
This follows from Lemma~\ref{le:perf} once we have shown that every
object in $\A$ has finite projective dimension, keeping in mind that
every object in $\Filt(\De_1,\ldots,\De_n)$ admits a
projective resolution in $\A$ that belongs to $\Filt(\De_1,\ldots,\De_n)$.

The fact that every object in $\A$ has finite projective dimension is
shown by induction on $n$. Consider
$\bar\A=\{X\in\A\mid\Hom_\A(\De_n,X)=0\}$ and for each $X$ in $\A$ the
exact sequence \eqref{eq:counit}. Then $\Ker\e_X$ and $\bar X$ have
finite projective dimension in $\bar\A$ since $\bar\A$ is a highest
weight category with $n-1$ standard objects, by
Lemma~\ref{le:reduction}. Every projective object from $\bar\A$ has
projective dimension at most one in $\A$ since it is of the form
$\bar P$ for some projective $P$ in $\A$ and $\e_P$ is a
monomorphism. Thus $\Ker\e_X$ and $\bar X$ have finite projective
dimension in $\A$. It follows that $X$ has finite projective
dimension.
\end{proof}

\begin{rem}
The proof of Lemma~\ref{le:filt} shows that $\Ext_\A^{2n-1}(-,-)=0$
for a highest weight category $\A$ with $n$ standard objects. This
bound is well known \cite{DR1989}.
\end{rem}

\begin{prop}\label{pr:hwt-except}
  The standard objects of a highest weight category form a strictly
  full exceptional sequence.
\end{prop}
\begin{proof}
  Let $\De_1,\ldots,\De_n$ be the exceptional objects.  It follows
  from Lemma~\ref{le:reduction} by induction on $n$ that the sequence
  $(\De_1,\ldots,\De_n)$ is exceptional. The sequence is strictly full
  by Lemma~\ref{le:filt}.
\end{proof}

The following lemma is the key for relating exceptional sequences and
highest weight categories; it is a variation of the `standardisation'
which Dlab and Ringel introduced in \cite{DR1992}.

\begin{lem}\label{le:standard}
  Let $\A$ be an abelian category and $(E_1,\ldots,E_n)$ a sequence
  of objects satisfying the following:
\begin{enumerate}
\item $\Ext^1_\A(E_i,E_j)=0$ for all $i\ge j$. 
\item $\Ext^1_\A(X,E_j)$ is finitely generated over
  $\End_\A(E_j)^\op$ for all $X\in\A$.
\end{enumerate}
Then there are exact sequences
\begin{equation}\label{eq:standard}
0\lto U_i\lto P_i\lto E_i \lto 0\qquad (1\le i\le n)
\end{equation}
in $\A$ such that $U_i$ belongs to $\Filt(E_{i+1},\ldots,E_n)$ for
all $i$ and $\bigoplus_{i=1}^n P_i$ is a projective generator of
$\Filt(E_1,\ldots,E_n)$.
\end{lem}
\begin{proof}
  We use induction on $n$.  The induction hypothesis yields a collection
  of exact sequences
  \[0\lto \bar U_i\lto \bar P_i\lto E_i \lto 0\qquad (1\le i < n)\]
  in $\Filt(E_{i+1},\ldots,E_{n-1})$. We modify them as
  follows. Using Ext-finiteness we can form the universal extension
  \[0\lto E_n^r\lto P_i\lto\bar P_i\lto 0\]
  in $\A$, that is, the induced map
  $\Hom_\A(E_n^r,E_n)\to\Ext_\A^1(\bar P_i,E_n)$ is
  surjective. This implies $\Ext^1_\A(P_i,E_n)=0$. Also,
  $\Ext^1_\A(P_i,-)$ vanishes on $\Filt(E_1,\ldots,E_{n-1})$. 

  We claim that $P_i$ is a projective object in
  $\Filt(E_1,\ldots,E_n)$.  First observe that each object $X$ in
  $\Filt(E_1,\ldots,E_n)$ fits into an exact sequence
  $0\to E_n^s\to X\to \bar X\to 0$ for some $s\ge 0$ with $\bar X$ in
  $\Filt(E_1,\ldots,E_{n-1})$, since $E_n$ is projective.  Now apply
  $\Ext^1_\A(P_i,-)$ to this sequence. 

  We obtain the following commutative diagram with exact rows and
  columns
\begin{equation*}
\begin{tikzcd}
{}&0\arrow{d}&0\arrow{d}\\
{}&E_n^r\arrow{d}\arrow[equal]{r}&E_n^r\arrow{d}\\
0\arrow{r}&U_i\arrow{d}\arrow{r}&P_i\arrow{d}\arrow{r}&E_i\arrow[equal]{d}\arrow{r}&0\\
0\arrow{r}&\bar U_i\arrow{d}\arrow{r}&\bar P_i\arrow{d}\arrow{r}&E_i\arrow{r}&0\\
{}&0&0
\end{tikzcd}
\end{equation*}
and get exact sequences \eqref{eq:standard}
with $U_i$ in $\Filt(E_{i+1},\ldots,E_n)$, where $P_n:=E_n$ and
$U_n:=0$. It remains to observe that $\bigoplus_i P_i$ is a
projective generator of $\Filt(E_1,\ldots,E_n)$.
\end{proof}

\begin{proof}[Proof of Theorem~\ref{th:exceptional}]
  (1) $\Rightarrow$ (2): Let $(E_1,\ldots, E_n)$ be an exceptional
  sequence in $\A$. Then it follows from Lemma~\ref{le:standard} that
  $\Filt(E_1,\ldots, E_n)$ admits a projective generator, say $P$. Set
  $\La=\End_\A(P)$ and $\De_i=\Hom_\A(P,E_i)$ for $1\le i\le n$. Then
  $\Hom_\A(P,-)$ induces a fully faithful and exact functor
  $\Filt(E_1,\ldots, E_n)\to\mod\La$, and $\mod\La$ is a highest
  weight category with standard objects $\De_1,\ldots,\De_n$ because
  of the sequences \eqref{eq:standard}. If $(E_1,\ldots, E_n)$ is
  strictly full, then $\Hom_\A(P,-)$ extends to a triangle equivalence
  $\bfD^b(\A)\xto{\sim} \bfD^b(\mod \La)$ by Lemma~\ref{le:filt}.

  (2) $\Rightarrow$ (1): Let $F\colon\bfD^b(\A)\xto{\sim} \bfD^b(\A')$
  be a triangle equivalence that identifies $(E_1,\ldots,E_n)$ with
  the sequence of standard objects $(\De_1,\ldots,\De_n)$ in
  $\A'$. Then the sequence $(E_1,\ldots,E_n)$ is exceptional, because
  $(\De_1,\ldots,\De_n)$ is exceptional by
  Proposition~\ref{pr:hwt-except}. An induction on $n$ shows that $F$
  induces an equivalence
  \[\Filt(E_1,\ldots, E_n)\xto{\ \sim\ }\Filt(\De_1,\ldots, \De_n).\]
  Here we use the fact that for each object $X$ in
  $\Filt(E_{1},\ldots,E_n)$ there is some $r\ge 0$ and an exact
  sequence $0\to E_n^r\to X\to X'\to 0$ with $X'$ in
  $\Filt(E_{1},\ldots,E_{n-1})$.  This equivalence extends to a
  triangle equivalence
  \[\bfD^b(\Filt(E_1,\ldots, E_n))\xto{\ \sim\
  }\bfD^b(\Filt(\De_1,\ldots, \De_n))\]
  making the following square of exact functors commutative
\[
\begin{tikzcd}
\bfD^b(\Filt(E_1,\ldots,E_n))\arrow{r}{\sim}\arrow{d}&\bfD^b(\Filt(\De_1,\ldots,\De_n)) \arrow{d}\\
\bfD^b(\A) \arrow{r}[below]{F}&\bfD^b(\A')
\end{tikzcd}
\]
where the vertical functors are induced by the inclusions
$\Filt(E_1,\ldots, E_n)\to\A$ and $\Filt(\De_{1},\ldots,\De_n)\to\A'$
respectively.  The vertical functor on the right is an equivalence by
Lemma~\ref{le:filt}, and it follows that the vertical functor on the
left is an equivalence. Thus the sequence $(E_1,\ldots,E_n)$ is
strictly full.
\end{proof}

I am grateful to Martin Kalck for providing the following
example of a full exceptional sequence that is not strictly full.

\begin{exm}[Kalck]\label{ex:kalck}
Fix a field $k$ and consider the finite dimensional $k$-algebra $\La$
given by the following quiver with relations.
\[
\begin{tikzcd}[column sep=normal, row sep=scriptsize]
&1 \arrow{ld}[swap]{\alpha}\\
2  \arrow{rr}[swap]{\beta}
  &&3 \arrow{lu}[swap]{\gamma}
\end{tikzcd}\qquad\qquad
\begin{aligned}
\gamma  \beta&= 0\\ 
\alpha \gamma&=0
\end{aligned}
\]
For each vertex $i$ let $S_i$ denote the  corresponding simple
$\La$-module and $P_i$ its projective cover. Then $(S_1,P_2,P_3)$ is
an exceptional sequence in $\A=\mod\La$ which generates $\bfD^b(\A)$
as a triangulated category. Set $\B=\Filt(S_1,P_2,P_3)$. Then we have
$\B=\add (S_1\oplus P_2\oplus P_3)$ but $\Ext^2_\La(S_1,P_2)\neq
0$. Thus the canonical functor $\bfD^b(\B)\to\bfD^b(\A)$ is not full.
\end{exm}

The following geometric example is more involved, and I am grateful to
Nathan Broomhead for allowing me to include this.

\begin{exm}[Broomhead]\label{ex:broomhead}
Let $\bbX$ be the blow up of $\mathbb P^3$ at a torus invariant
point. We consider this as a toric variety given by a fan with rays:
\[\{(1,0,0),(0,1,0),(0,0,1),(-1,-1,-1), (1,1,1)\}.\]
Label the corresponding divisors $D_1,D_2,D_3,D,E$. Note that $D$ and
$E$ form a basis of $\Pic \bbX$, where $D_i\sim D-E$ for $i=1,2,3$. An
explicit calculation shows that
\[X=(\O(-3D+2E),\O(-2D+E),\O(-D),\O(-2D+2E),\O(-D+E),\O)\]
is a full strong exceptional sequence in $\A=\coh\bbX$. Mutating this
sequence, we obtain a new full exceptional sequence
\[X'=(\O(-2D+E),\O(-D),\O(-E),\O(-2D+2E),\O(-D+E),\O)\]
which is not strictly full. Set $\B=\Filt(X')$. Then
$\B=\add(\O(-2D+E)\oplus\O(-D)\oplus\O(-E)\oplus\O(-2D+2E)\oplus\O(-D+E)\oplus\O)$
but $\Ext^2_\bbX(\O(-E),\O(-2D+2E))\neq 0$. Thus the canonical functor
$\bfD^b(\B)\to\bfD^b(\A)$ is not full.
\end{exm}

We end this note by giving the proof of Theorem~\ref{th:standard-defn}
from the introduction.

\begin{proof}[Proof of Theorem~\ref{th:standard-defn}]
  Let $\A=\mod\La$ for some artin algebra. Suppose first that $\A$ is
  a highest weight category with standard objects
  $\De_1,\ldots,\De_n$. Then all but one of the conditions (1)--(4)
  hold by definition, while (3) follows by induction on $n$ from
  Lemma~\ref{le:reduction}.  The converse is an immediate consequence
  of Lemma~\ref{le:standard}.
\end{proof}

\begin{appendix}

\section{Homological recollements}

There are well known criteria for an inclusion of abelian categories
$\A'\to\A$ to extend to a fully faithful functor between their derived
categories \cite{Ha1966,Il1971, Ke1996}, and closely related is the
question when the inclusion induces isomorphisms
\begin{equation*}\label{eq:ext-iso}
\Ext^p_{\A'}(X,Y)\xto{\sim}\Ext^p_\A(X,Y)
\end{equation*}
for all $X,Y\in\A'$ and $p\ge 0$.

This appendix provides a necessary and sufficient criterion for a
colocalisation sequence of abelian categories
\begin{equation}\label{eq:coloc-app}
\begin{tikzcd}
\A' \arrow[tail,yshift=-0.75ex]{rr}[swap]{i_*} &&\A  \arrow[twoheadrightarrow,yshift=0.75ex]{ll}[swap]{i^*}
\arrow[twoheadrightarrow,yshift=-0.75ex]{rr}[swap]{j^!} &&\A''\, .
\arrow[tail,yshift=0.75ex]{ll}[swap]{j_!}
\end{tikzcd}
\end{equation}

\begin{prop}\label{pr:derived}
  Suppose that $\A$ has enough projective objects and that $j^!$
  preserves projectivity.  Then the following are equivalent:
\begin{enumerate}
\item The counit $j_!j^!(X)\to X$ is a
  monomorphism for every projective  $X\in\A$.
\item There is an induced colocalisation sequence of triangulated
  categories
\begin{equation}\label{eq:coloc-app-der}
\begin{tikzcd}
  \bfD^{-}(\A') \arrow[tail,yshift=-0.75ex]{rr}[swap]{i_*}
  &&\bfD^{-}(\A)
  \arrow[twoheadrightarrow,yshift=0.75ex]{ll}[swap]{i^*}
  \arrow[twoheadrightarrow,yshift=-0.75ex]{rr}[swap]{j^!}
  &&\bfD^{-}(\A'')\, .  \arrow[tail,yshift=0.75ex]{ll}[swap]{j_!}
\end{tikzcd}
\end{equation}
\end{enumerate}
\end{prop}
\begin{proof}
  (1) $\Rightarrow$ (2): Let $\P$ denote the full subcategory of
  projective objects in $\A$; the categories $\P'$ and $\P''$ are
  defined analogously. We view $\A'$ and $\A''$ as full subcategories
  of $\A$ via $i_*$ and $j_!$, respectively, and write
  $\Filt(\P',\P'')$ for the smallest extension closed subcategory of
  $\A$ containing $\P'$ and $\P''$. This contains $\P$ since each
  projective object $X$ fits into an exact sequence
  \[0\lto j_!j^!(X)\lto X\lto i_*i^*(X)\lto 0.\]
  Note that the diagram \eqref{eq:coloc-app} restricts to
\begin{equation}\label{eq:coloc-app-proj}
\begin{tikzcd}
\P' \arrow[tail,yshift=-0.75ex]{rr}[swap]{i_*} &&\Filt(\P',\P'')  \arrow[twoheadrightarrow,yshift=0.75ex]{ll}[swap]{i^*}
\arrow[twoheadrightarrow,yshift=-0.75ex]{rr}[swap]{j^!} &&\P''  \arrow[tail,yshift=0.75ex]{ll}[swap]{j_!}
\end{tikzcd}
\end{equation}
and all functors in this diagram are exact. The only functor for which
this is not obvious is $i^*$. In that case exactness follows from the
snake lemma because the counit $j_!j^!(X)\to X$ is a monomorphism for
every $X$ in $\Filt(\P',\P'')$. Thus the diagram
\eqref{eq:coloc-app-proj} induces a colocalisation sequence
\begin{equation}\label{eq:coloc-app-proj-der}
\begin{tikzcd}
  \bfD^-(\P') \arrow[tail,yshift=-0.75ex]{rr}[swap]{i_*}
  &&\bfD^-(\Filt(\P',\P''))
  \arrow[twoheadrightarrow,yshift=0.75ex]{ll}[swap]{i^*}
  \arrow[twoheadrightarrow,yshift=-0.75ex]{rr}[swap]{j^!}
  &&\bfD^-(\P'')\, .
  \arrow[tail,yshift=0.75ex]{ll}[swap]{j_!}
\end{tikzcd}
\end{equation}
We claim that the diagrams \eqref{eq:coloc-app-der} and
\eqref{eq:coloc-app-proj-der} are equivalent via triangle equivalences
induced by the inclusions 
\[f'\colon\P'\to\A'\qquad f''\colon\P''\to\A''\qquad
f\colon\Filt(\P',\P'')\to\A.\]
This is clear for $f'$ and $f''$, since $\A'$ and $\A''$ have enough
projective objects. For $f$ it suffices to note that the
inclusion $\P\to\Filt(\P',\P'')$ yields a triangle equivalence
$\bfD^-(\P)\xto{\sim}\bfD^-(\Filt(\P',\P''))$, since $\P$ equals the
full subcategory of projective objects of the exact category
$\Filt(\P',\P'')$.

(2) $\Rightarrow$ (1): Suppose there is a colocalisation sequence
\eqref{eq:coloc-app-der}. Given a projective object $X$ in $\A$, we
have an exact triangle \[ j_!j^!(X)\lto X\lto i_*i^*(X)\lto \]
in $\bfD^-(\A)$. This uses the fact that for complexes of projectives
the derived functors of $i^*$ and $j_!$ are defined degreewise via
$i^*$ and $j_!$, respectively. Taking cohomology, we obtain an exact
sequence
\[\cdots\lto 0\lto j_!j^!(X)\lto X\lto i_*i^*(X)\lto 0 \lto \cdots\]
in $\A$. It follows that the counit $j_!j^!(X)\to X$ is a
monomorphism.
\end{proof}

\begin{rem}
  There is a dual version of Proposition~\ref{pr:derived} for
  localisation sequences of abelian categories with enough injective
  objects. This situation arises frequently, a typical example being a
  Grothendieck abelian category $\A$ with localising subcategory
  $\A'$.
\end{rem}

\begin{rem} 
  Proposition~\ref{pr:derived} covers a couple of known criteria for
  an inclusion of abelian categories $\A'\to\A$ to extend to a fully
  faithful functor between their derived categories. Consider a
  colocalisation sequence \eqref{eq:coloc-app} and suppose that $\A$
  has enough projective objects.

  (1) The criterion in \cite[Proposition~4.8]{Ha1966} requires that
  every object $Y$ in $\A'$ admits an epimorphism $X\to Y$ in $\A'$
  with $X$ projective in $\A$. This implies easily that $j^!$
  preserves projectivity and that for each projective $X$ in $\A$ the
  counit $j_!j^!(X)\to X$ is a \emph{split} monomorphism.

  (2) The criterion in \cite[\S 12]{Ke1996} requires for every
  epimorphism $X\to Y$ in $\A$ with $Y$ in $\A'$ the existence of an
  epimorphism $X'\to Y$ in $\A'$ that factors through $X\to Y$.  Given
  our assumptions, this condition is equivalent to the one in
  \cite[Proposition~4.8]{Ha1966}.
\end{rem}

\end{appendix}

\end{document}